\documentclass[12pt]{article}
 \usepackage[cp1251]{inputenc} 
 \usepackage[english,russian]{babel}

\usepackage{amsmath}
\usepackage{amssymb}
\usepackage{amsxtra}
\usepackage{latexsym}
\usepackage{ifthen}

\begin{document}

\newcounter{lemma}
\newcommand{\lemma}{\par \refstepcounter{lemma}%
{\bf ╦хььр \arabic{lemma}.}}

\newcounter{corollary}
\newcommand{\corollary}{\par \refstepcounter{corollary}%
{\bf ╤ыхфёЄтшх \arabic{corollary}.}}

\newcounter{remark}
\newcommand{\remark}{\par \refstepcounter{remark}%
{\bf ╟рьхўрэшх \arabic{remark}.}}

\newcounter{theorem}
\newcommand{\theorem}{\par \refstepcounter{theorem}%
{\bf ╥хюЁхьр \arabic{theorem}.}}

\newcounter{proposition}
\newcommand{\proposition}{\par \refstepcounter{proposition}%
{\bf ╧Ёхфыюцхэшх \arabic{proposition}.}}

\renewcommand{\refname}{\centerline{\bf ╤яшёюъ ышЄхЁрЄєЁ√}}

\newcommand{\proof}{{\it ─юърчрЄхы№ёЄтю.\,\,}}

\title{ {\leftline{\small ╙─╩ 517.5}}
\medskip ╬ эшцэхь яюЁ фъх юЄюсЁрцхэшщ ё ъюэхўэ√ь шёърцхэшхь фышэ√}
\author{╤хтюёЄ№ эют ┼.└.}
\maketitle
%
%



%
%
%
%
\maketitle
\begin{abstract}
┬ эрёЄю ∙хщ ЁрсюЄх шчєўрхЄё  тюяЁюё ю Єръ эрч√трхьюь эшцэхь яюЁ фъх
фы  юфэюую яюфтшфр юЄюсЁрцхэшщ ё ъюэхўэ√ь шёърцхэшхь, ръЄштэю
шчєўрхь√ї яюёыхфэшх 15--20 ыхЄ. ─юърчрэю, ўЄю юЄюсЁрцхэш  ё ъюэхўэ√ь
шёърцхэшхь фышэ√ $f:D\rightarrow {\Bbb R}^n,$ $n\ge 2,$ тэх°э  
фшырЄрЎш  ъюЄюЁ√ї ыюъры№эю ёєььшЁєхьр т ёЄхяхэш $\alpha>n-1$ ш
шьх■∙шх ъюэхўэ√щ рёшьяЄюЄшўхёъшщ яЁхфхы, шьх■Є ЁртэюьхЁэю
юуЁрэшўхээ√щ ёэшчє эшцэшщ яюЁ фюъ.
\end{abstract}

{\it ╩ы■ўхт√х ёыютр:} юЄюсЁрцхэш  ё юуЁрэшўхээ√ь ш ъюэхўэ√ь
шёърцхэшхь, ЁюёЄ юЄюсЁрцхэш  эр схёъюэхўэюёЄш, юЄъЁ√Є√х фшёъЁхЄэ√х
юЄюсЁрцхэш , ╕ьъюёЄш ъюэфхэёрЄюЁют


\section{┬тхфхэшх} ═шцэшщ яюЁ фюъ юЄюсЁрцхэшщ, ъръ шчтхёЄэю, шуЁрхЄ
ёє∙хёЄтхээє■ Ёюы№ яЁш фюърчрЄхы№ёЄтх эхъюЄюЁ√ї рэрыюуют ЄхюЁхь√
╧шърЁр, р Єръцх т ЄхюЁшш ЁрёяЁхфхыхэш  чэрўхэшщ (ёь., эряЁ.,
\cite{RV}, \cite{Ri} ш \cite{Ra$_1$}). ═рёЄю ∙р  чрьхЄър яюёт ∙хэр
шчєўхэш■ юЄюсЁрцхэшщ ё ъюэхўэ√ь шёърцхэшхь, ръЄштэю шчєўрхь√ї т
яюёыхфэхх тЁхь  (ёь., эряЁ., \cite{IM}, \cite{MRSY}--\cite{MRSY$_1$}
ш \cite{GRSY}). ╟фхё№ шёёыхфєхЄё  тюяЁюё юс юуЁрэшўхээюёЄш ёэшчє
эшцэхую яюЁ фър юЄюсЁрцхэшщ ё ъюэхўэ√ь шёърцхэшхь фышэ√, ттхф╕ээ√ї
╬. ╠рЁЄшю, ┬. ╨ чрэют√ь, ╙. ╤ЁхсЁю ш ▌. ▀ъєсют√ь
\cite{MRSY}--\cite{MRSY$_1$}.

┬ ёЁртэшЄхы№эю эхфртэхщ ЁрсюЄх ╩.~╨рщрыр \cite{Ra$_1$} с√ыю
єёЄрэютыхэю, ўЄю юЄюсЁрцхэш  ё ъюэхўэ√ь шёърцхэшхь, юяЁхфхы╕ээ√х тю
тё╕ яЁюёЄЁрэёЄтх ${\Bbb R}^n,$ $n\ge 2,$ яЁш ёююЄтхЄёЄтє■∙шї
юуЁрэшўхэш ї эр Єръ эрч√трхь√щ эшцэшщ яюЁ фюъ юЄюсЁрцхэш  (яютхфхэшх
юЄюсЁрцхэш  эр схёъюэхўэюёЄш), р Єръцх ЄЁхсютрэш ї эр тэх°э■■
фшырЄрЎш■, юсырфр■Є эхъюЄюЁ√ьш чрёыєцштр■∙шьш тэшьрэш  ётющёЄтрьш.
(╬яЁхфхыхэшх ш яЁшьхЁ√ юЄюсЁрцхэшщ ё ъюэхўэ√ь шёърцхэшхь ьюуєЄ с√Є№
эрщфхэ√ т ьюэюуЁрЇшш \cite{IM}). ├ютюЁ  Єюўэхх, т ЁрсюЄх
\cite{Ra$_1$} с√ыю яюърчрэю, ўЄю єърчрээ√х юЄюсЁрцхэш  шьх■Є
юуЁрэшўхээ√щ ёэшчє эшцэшщ яюЁ фюъ, Є.х., юуЁрэшўхэр ёэшчє эхъюЄюЁр 
тхышўшэр, юЄтхўр■∙р  чр ЁюёЄ юЄюсЁрцхэш . ╬ЄьхЄшь, ўЄю яюфюсэ√щ
Ёхчєы№ЄрЄ юсюс∙рхЄ ъырёёшўхёъє■ ЄхюЁхьє ╨шъьрэр--┬єюЁшэхэр,
юЄэюё ∙є■ё  ъ ъырёёє юЄюсЁрцхэшщ ё юуЁрэшўхээ√ь шёърцхэшхь
(\cite{RV}).

\medskip
┬ эрёЄю ∙хщ ёЄрЄ№х єёЄрэртыштрхЄё  эхъюЄюЁ√щ рэрыюу Ёхчєы№ЄрЄют
ЁрсюЄ√ \cite{Ra$_1$} фы  Єръ эрч√трхь√ї юЄюсЁрцхэшщ ё ъюэхўэ√ь
шёърцхэшхь фышэ√ (\cite[уы.~8]{MRSY}) яЁш эхёъюы№ъю шэ√ї
яЁхфяюыюцхэш ї эр шї тэх°э■■ фшырЄрЎш■. ╧Ёш ¤Єюь, ёююЄтхЄёЄтє■∙шх
яЁшьхЁ√, яЁштхф╕ээ√х т яюёыхфэхь Ёрчфхых, яюърч√тр■Є, ўЄю єърчрээ√х
єёыютш  яЁшэЎшяшры№эю юЄышўр■Єё  юЄ юуЁрэшўхэшщ шч ЁрсюЄ√
\cite{Ra$_1$}. ╬ЄюсЁрцхэш  ё ъюэхўэ√ь шёърцхэшхь фышэ√ ттхфхэ√ ╬.
╠рЁЄшю ёютьхёЄэю ё ┬. ╨ чрэют√ь, ╙. ╤ЁхсЁю ш ▌. ▀ъєсют√ь
(\cite{MRSY$_1$}) ш яЁхфёЄрты ■Є ёюсющ юфэю шч юсюс∙хэшщ юЄюсЁрцхэшщ
ё юуЁрэшўхээ√ь шёърцхэшхь яю ╨х°хЄэ ъє (ёь. \cite{Re} ш \cite{Ri}).
╬эш ьюуєЄ с√Є№ юяЁхфхыхэ√ ъръ юЄюсЁрцхэш , шёърцр■∙шх хтъышфютю
ЁрёёЄю эшх т ъюэхўэюх ўшёыю Ёрч т яюўЄш тёхї Єюўърї, р Єръцх
юсырфр■∙шх $N$-ётющёЄтюь ╦єчшэр юЄэюёшЄхы№эю ьхЁ√ ╦хсхур т ${\Bbb
R}^n$ ш ьхЁ√ фышэ√ эр ъЁшт√ї т яЁ ьє■ ш юсЁрЄэє■ ёЄюЁюэ√ (ёь. Єрь
цх).

\medskip
┬ё■фє фрыхх $D$ -- юсырёЄ№ т ${\Bbb R}^n,$ $n\ge 2,$ $m$ -- ьхЁр
╦хсхур ${\Bbb R}^n,$ ${\rm dist\,}(A,B)$ -- хт\-ъыш\-фютю ЁрёёЄю эшх
ьхцфє ьэюцхёЄтрьш $A, B\subset {\Bbb R}^n,$ ${\rm
dist\,}(A,B)=\inf\limits_{x\in A, y\in B} |x-y|,$ $(x,y)$ юсючэрўрхЄ
(ёЄрэфрЁЄэюх) ёъры Ёэюх яЁюшчтхфхэшх тхъЄюЁют $x,y\in {\Bbb R}^n,$
${\rm diam\,}A$ -- хт\-ъыш\-фют фшрьхЄЁ ьэюцхёЄтр $A\subset {\Bbb
R}^n,$ $B(x_0, r)=\left\{x\in{\Bbb R}^n: |x-x_0|< r\right\},$ ${\Bbb
B}^n := B(0, 1),$ $S(x_0,r) = \{ x\,\in\,{\Bbb R}^n : |x-x_0|=r\},$
${\Bbb S}^{n-1}:=S(0, 1),$ $\omega_{n-1}$ ючэрўрхЄ яыю∙рф№ ёЇхЁ√
${\Bbb S}^{n-1}$ т ${\Bbb R}^n,$ $\Omega_{n}$ -- юс·╕ь хфшэшўэюую
°рЁр ${\Bbb B}^{n}$ т ${\Bbb R}^n,$ чряшё№ $f:D\rightarrow {\Bbb
R}^n$ яЁхфяюырурхЄ, ўЄю юЄюсЁрцхэшх $f,$ чрфрээюх т юсырёЄш $D,$
эхяЁхЁ√тэю. \medskip {\it ╩Ёштющ} $\gamma$ ь√ эрч√трхь эхяЁхЁ√тэюх
юЄюсЁрцхэшх юЄЁхчър $[a,b]$ (юЄъЁ√Єюую шэЄхЁтрыр $(a,b),$ ышсю
яюыєюЄъЁ√Єюую шэЄхЁтрыр тшфр  $[a,b)$ шыш $(a,b]$) т ${\Bbb R}^n,$
$\gamma:[a,b]\rightarrow {\Bbb R}^n.$ ╧юф ёхьхщёЄтюь ъЁшт√ї $\Gamma$
яюфЁрчєьхтрхЄё  эхъюЄюЁ√щ ЇшъёшЁютрээ√щ эрсюЁ ъЁшт√ї $\gamma,$ р
$f(\Gamma)=\left\{f\circ\gamma|\gamma\in\Gamma\right\}.$ ─рыхх
ёшьтюы $\Gamma(E,F,D)$ ючэрўрхЄ ёхьхщёЄтю тёхї ъЁшт√ї
$\gamma:[a,b]\rightarrow{\Bbb R}^n,$ ъюЄюЁ√х ёюхфшэ ■Є $E$ ш $F$ т
$D,$ Є.х. $\gamma(a)\in E,$ $\gamma(b)\in F$ ш $\gamma(t)\in D$ яЁш
$t\in (a, b).$  ╤ыхфє■∙шх юяЁхфхыхэш  ьюуєЄ с√Є№ эрщфхэ√, эряЁ., т
Ёрчф. 1--6 уы. I т \cite{Va$_1$}. ┴юЁхыхтр ЇєэъЎш  $\rho:{\Bbb
R}^n\,\rightarrow [0,\infty]$ эрч√трхЄё  {\it фюяєёЄшьющ} фы 
ёхьхщёЄтр $\Gamma$ ъЁшт√ї $\gamma$ т ${\Bbb R}^n,$ хёыш
ъЁштюышэхщэ√щ шэЄхуЁры яхЁтюую Ёюфр юЄ ЇєэъЎшш $\rho$ яю ърцфющ
(ыюъры№эю ёяЁ ьы хьющ) ъЁштющ $\gamma\in \Gamma$ єфютыхЄтюЁ хЄ
єёыютш■ $\int\limits_{\gamma}\rho (x)|dx|\ge 1.$ ┬ ¤Єюь ёыєўрх ь√
яш°хь: $\rho \in {\rm adm}\Gamma.$ {\it ╠юфєыхь} ёхьхщёЄтр ъЁшт√ї
$\Gamma $ эрч√трхЄё  тхышўшэр $M(\Gamma)=\inf_{\rho \in \,{\rm
adm}\,\Gamma} \int\limits_D \rho ^n (x)\, dm(x).$ ╤тющёЄтр ьюфєы  т
эхъюЄюЁющ ьхЁх рэрыюушўэ√ ётющёЄтрь ьхЁ√ ╦хсхур $m$ т ${\Bbb R}^n$
(ёь. \cite[ЄхюЁхьр~6.2]{Va$_1$}). ╧єёЄ№ $Q(x):D\rightarrow [0,
+\infty]$ -- шчьхЁшьр  яю ╦хсхує ЇєэъЎш . ─ы  яЁюшчтюы№эюую
ёхьхщёЄтр ъЁшт√ї $\Gamma$ тхышўшэр $M_{Q}(\Gamma),$ эрч√трхьр  {\it
ьюфєыхь ёхьхщёЄтр $\Gamma$ ё тхёюь $Q$,} ьюцхЄ с√Є№ юяЁхфхыхэр
ёююЄэю°хэшхь
$M_{Q}(\Gamma):=\inf\limits_{\rho\in {\rm
adm\,}\Gamma}\int\limits_{{\Bbb R}^n} Q(x)\cdot\rho^{n}(x)dm(x).$
╧юырурхь
$$L(x,\varphi)=\limsup\limits_{y\rightarrow x, y \in
E}\,\frac{|\varphi(x)-\varphi(y)|}{|y-x|}\,,\quad
l(x,\varphi)=\liminf\limits_{y\rightarrow x, y\in E}\,
\frac{|\varphi(x)-\varphi(y)|}{|y-x|}\,.$$

\medskip ╬яЁхфхышь ЇєэъЎшш фышэ√, ёт чрээ√х ё юЄюсЁрцхэшхь эр ъЁшт√ї
(ёь. \cite{MRSY} ш \cite{MRSY$_1$}). ╧єёЄ№ $\Delta \subset \Bbb R$
-- юЄъЁ√Є√щ шэЄхЁтры ўшёыютющ яЁ ьющ, $\gamma: \Delta\rightarrow
{\Bbb R}^n$ -- ыюъры№эю ёяЁ ьы хьр  ъЁштр . ┬ Єръюь ёыєўрх,
юўхтшфэю, ёє∙хёЄтєхЄ хфшэёЄтхээр  эхєс√тр■∙р  ЇєэъЎш  фышэ√
$l_{\gamma}:\Delta\rightarrow \Delta_{\gamma}\subset \Bbb{R}$ ё
єёыютшхь $l_{\gamma}(t_0)=0,$ $t_0 \in \Delta,$ Єрър  ўЄю чэрўхэшх
$l_{\gamma}(t)$ Ёртэю фышэх яюфъЁштющ $\gamma\mid_{[t_0, t]}$ ъЁштющ
$\gamma,$ хёыш $t>t_0,$ ш фышэх яюфъЁштющ $\gamma\mid_ {[t,\,t_0]}$
ёю чэръюь $"$$-$$"$, хёыш $t<t_0,$ $t\in \Delta.$ ╧єёЄ№
$g:|\gamma|\rightarrow {\Bbb R}^n$ -- эхяЁхЁ√тэюх юЄюсЁрцхэшх, уфх
$|\gamma| = \gamma(\Delta)\subset \Bbb{R}^n.$ ╧Ёхфяюыюцшь, ўЄю
ъЁштр  $\widetilde{\gamma}=g\circ \gamma$ Єръцх ыюъры№эю ёяЁ ьы хьр.
╥юуфр, юўхтшфэю, ёє∙хёЄтєхЄ хфшэёЄтхээр  эхєс√тр■∙р  ЇєэъЎш 
$L_{\gamma,\,g}:\,\Delta_{\gamma} \rightarrow
\Delta_{\widetilde{\gamma}}$ Єрър , ўЄю
$L_{\gamma,\,g}\left(l_{\gamma}\left(t\right)
 \right)\,=\,l_{\widetilde{\gamma}}\left(t\right)$ яЁш тёхї
$t\in\Delta.$ ╤юуырёэю \cite[уы.~8]{MRSY} (ёь. Єръцх
\cite{MRSY$_1$}), юЄюсЁрцхэшх $f:D\rightarrow {\Bbb R}^n$ сєфхь
эрч√трЄ№ {\it юЄюсЁрцхэшхь ё ъюэхўэ√ь шёърцхэшхь фышэ√} (яш°хь:
$f\in FLD$), хёыш $f$ юсырфрхЄ $N$-ётющёЄтюь ╦єчшэр, фы  я.т. $x\in
D$
$0<l(x, f)\le L (x, f)<\infty$
ш, ъЁюьх Єюую, т√яюыэхэ√ ёыхфє■∙шх єёыютш :

\medskip
$(L_1)$\quad фы  я.т. ъЁшт√ї $\gamma\in D$\,\, ъЁштр 
$\widetilde{\gamma}=f \circ \gamma$ ыюъры№эю ёяЁ ьы хьр ш ЇєэъЎш 
$L_{\gamma,\,f}$ юсырфрхЄ $N$-ётющёЄтюь;

\medskip
$(L_2)$\quad фы  я.т. ъЁшт√ї $\widetilde{\gamma} \in f(D)$ ърцфюх
яюфэ Єшх $\gamma$ ъЁштющ $\widetilde{\gamma}$ ыюъры№эю ёяЁ ьы хью ш
ЇєэъЎш  $L_{\gamma,\,f}$ юсырфрхЄ $N^{\,-1}$-ётющёЄтюь.

\medskip ╟фхё№ фрыхх ъЁштр  $\gamma \in D$ эрч√трхЄё  {\it (яюыэ√ь) яюфэ Єшхь ъЁштющ
$\widetilde{\gamma}\in {\Bbb R}^n$ яЁш юЄюсЁрцхэшш $f:D \rightarrow
{\Bbb R}^n,$} хёыш $\widetilde{\gamma}=f \circ \gamma.$ ├ютюЁ Є, ўЄю
эхъюЄюЁюх ётющёЄтю т√яюыэхэю фы  {\it яюўЄш тёхї (я.т.) ъЁшт√ї}
юсырёЄш $D$, хёыш юэю шьххЄ ьхёЄю фы  тёхї ъЁшт√ї, ыхцр∙шї т $D$,
ъЁюьх эхъюЄюЁюую шї ёхьхщёЄтр, ьюфєы№ ъюЄюЁюую Ёртхэ эєы■.

\medskip
\begin{remark}\label{rem1}
╤юуырёэю \cite[ёыхфёЄтшх~8.1]{MRSY}, ёь. Єръцх
\cite[ёыхфёЄтшх~3.14]{MRSY$_1$}, юЄюсЁрцхэшх $f:D\rightarrow {\Bbb
R}^n$ яЁшэрфыхцшЄ ъырёёє юЄюсЁрцхэшщ ё ъюэхўэ√ь шёърцхэшхь фышэ√
Єюуфр ш Єюы№ъю Єюуфр, ъюуфр $f$ фшЇЇхЁхэЎшЁєхью яюўЄш тё■фє,
юсырфрхЄ $N$ ш $N^{\,-1}$-ётющёЄтрьш ш, ъЁюьх Єюую, т√яюыэхэ√
єёыютш  $(L_1)$ ш $(L_2).$ ┬ ётю■ юўхЁхф№, фы  фшёъЁхЄэ√ї юЄъЁ√Є√ї
юЄюсЁрцхэшщ єёыютш  $(L_1)$ ш $(L_2)$ ¤ътштрыхэЄэ√ ыюъры№эющ
рсёюы■Єэющ эхяЁхЁ√тэюёЄш ЇєэъЎшш фышэ√ $L_{\gamma,\,f},$
ёююЄтхЄёЄтхээю, т яЁ ьє■ ш юсЁрЄэє■ ёЄюЁюэ√ (ўЄю чряшё√тр■Є т тшфх
$f\in ACP$ ш $f\in ACP^{\,-1},$ ёь. эряЁ.,
\cite[яЁхфыюцхэшх~8.5]{MRSY}).
\end{remark}

\medskip
╧юырурхь $M_f(r)=\sup\limits_{x\in B(0, r)}|f(x)|.$ ╤юуырёэю
\cite{RV}, {\it эшцэшь яюЁ фъюь} юЄюсЁрцхэш  $f:{\Bbb
R}^n\rightarrow{\Bbb R}^n$ сєфхь эрч√трЄ№ тхышўшэє
%
$$\lambda_f:=\liminf\limits_{r\rightarrow\infty}(n-1)\frac{\log\log
M_f(r)}{\log r}$$
%
(ёь. Єръцх \cite{Ra$_1$}). ┼ёыш $\lambda_f>0,$ Єю сєфхь уютюЁшЄ№,
ўЄю $f$ шьххЄ яюыюцшЄхы№э√щ эшцэшщ яюЁ фюъ.

\medskip
═ряюьэшь, ўЄю ¤ыхьхэЄ $b\in {\Bbb R}^n$ {\it эрч√трхЄё 
рёшьяЄюЄшўхёъшь яЁхфхыюь} юЄюсЁрцхэш  $f$ т схёъюэхўэю єфры╕ээющ
Єюўъх, хёыш эрщф╕Єё  ъЁштр  $\gamma:[0, 1)\rightarrow { \Bbb R}^n$
Єрър , ўЄю $\lim\limits_{t\rightarrow 1-0}\gamma(t)=\infty$ ш
$\lim\limits_{t\rightarrow 1-0}f(\gamma(t))=b.$

\medskip
╩ръ юс√ўэю, фы  фшЇЇхЁхэЎшЁєхьюую т Єюўъх $x_0\in D$ юЄюсЁрцхэш 
$f:D\rightarrow {\Bbb R}^n$ ёшьтюы $f^{\,\prime}(x_0)$ юсючэрўрхЄ
ьрЄЁшЎє ▀ъюсш юЄюсЁрцхэш  $f$ т Єюўъх $x_0,$ р $J(x_0, f)$ --
 ъюсшрэ юЄюсЁрцхэш  $f$ т Єюўъх $x_0.$ ╧юырурхь $\Vert
f^{\,\prime}(x)\Vert=\max\limits_{h\in {\Bbb R}^n \setminus \{0\}}
\frac {|f^{\,\prime}(x)h|}{|h|}.$ {\it ┬эх°э   фшырЄрЎш } $K_O(x,f)$
(фшЇЇхЁхэЎшЁєхьюую) юЄюсЁрцхэш  $f$ т Єюўъх $x$ юяЁхфхы хЄё 
ёююЄэю°хэшхь
%
$$K_O(x,f)\quad =\quad\left\{
\begin{array}{rr}
\frac{\Vert f^{\,\prime}(x)\Vert^n}{|J(x, f)|}, & J(x,f)\ne 0,\\
1,  &  f^{\,\prime}(x)=0, \\
\infty, & \text{т\,\,юёЄры№э√ї\,\,ёыєўр ї}
\end{array}
\right.\,.$$
%
─ы  яЁюшчтюы№эющ шчьхЁшьющ яю ╦хсхує ЇєэъЎшш $Q:{\Bbb
R}^n\rightarrow [0,\infty]$ яюырурхь
\begin{equation}\label{eq32*}
q_{x_0}(r):=\frac{1}{\omega_{n-1}r^{n-1}}\int\limits_{|x-x_0|=r}Q(x)\,d{{\cal
H}^{n-1}(x)}\,.
\end{equation}
╬фэшь шч уыртэ√ї Ёхчєы№ЄрЄют эрёЄю ∙хщ ЁрсюЄ√  ты хЄё  єЄтхЁцфхэшх,
яЁштхф╕ээюх эшцх. ┼ую рэрыюу фюърчрэ т ёЄрЄ№х
\cite[ЄхюЁхьр~1.3]{Ra$_1$} фы  юЄюсЁрцхэшщ ё ъюэхўэ√ь шёърцхэшхь,
тэх°э   фшырЄрЎш  ъюЄюЁюую єфютыхЄтюЁ хЄ эхъюЄюЁ√ь єёыютш ь
¤ъёяюэхэЎшры№эюую ЁюёЄр. (╧ю ¤Єюьє яютюфє ёь. Єръцх ъырёёшўхёъшщ
ёыєўрщ юЄюсЁрцхэшщ ё юуЁрэшўхээ√ь шёърцхэшхь, ЁрёёьюЄЁхээ√щ
╤.~╨шъьрэюь ш ╠.~┬єюЁшэхэюь \cite{RV}).

\medskip
\begin{theorem}\label{th3}
{\sl ╧єёЄ№ $f:{\Bbb R}^n\rightarrow {\Bbb R}^n$ -- юЄъЁ√Єюх
фшёъЁхЄэюх юЄюсЁрцхэшх ё ъюэхўэ√ь шёърцхэшхь фышэ√. ╧Ёхфяюыюцшь,
ўЄю:

1) тэх°э   фшырЄрЎш  $K_O(x, f)$ юЄюсЁрцхэш  $f$ яЁш яюўЄш тёхї $x$
єфютыхЄтюЁ хЄ єёыютш■: $K_O(x, f)\le Q(x)\in L_{loc}^{\alpha}({\Bbb
R}^n),$ уфх $\alpha>n-1$ -- эхъюЄюЁюх ЇшъёшЁютрээюх ўшёыю;

2) эрщф╕Єё  $R_0>0$ Єръюх, ўЄю
\begin{equation}\label{eq18}
\frac{1}{m(B(0, t))}\int\limits_{B(0, t)} Q^{\alpha}(x)\,dm(x)\le
A\end{equation}
фы  эхъюЄюЁюую $A<\infty$ ш тёхї $t\ge R_0;$

3) яЁш $K\rightarrow\infty$
\begin{equation}\label{eq16}
\int\limits_{r}^{Kr}\frac{d\omega}{\omega
\widetilde{q}^{\frac{1}{n-1}}_0(\omega)}\rightarrow\infty
\end{equation}
ЁртэюьхЁэю яю $r> R_0,$ уфх $\widetilde{q}_0(\omega)$ юсючэрўрхЄ
ёЁхфэхх чэрўхэшх ЇєэъЎшш $Q^{n-1}$ эрф ёЇхЁющ $S(0, \omega).$

╥юуфр, хёыш юЄюсЁрцхэшх $f$ шьххЄ рёшьяЄюЄшўхёъшщ яЁхфхы $a\in {\Bbb
R}^n$ т схёъюэхўэю єфры╕ээющ Єюўъх, Єю хую эшцэшщ яюЁ фюъ
$\lambda_f$ єфютыхЄтюЁ хЄ єёыютш■
$$\lambda_f>M=M(n, Q)>0\,,$$
уфх ўшёыю $M$ чртшёшЄ Єюы№ъю юЄ ЁрчьхЁэюёЄш яЁюёЄЁрэёЄтр $n$ ш
ЇєэъЎшш $Q.$}
\end{theorem}

\medskip
\begin{corollary}\label{cor1}
{\sl ╟ръы■ўхэшх ЄхюЁхь√ \ref{th3} юёЄр╕Єё  т√яюыэхээ√ь, хёыш тьхёЄю
єёыютшщ (\ref{eq18}) ш (\ref{eq16}) т√яюыэхэю єёыютшх: $q_{\alpha,
0}(r)\le C $ яЁш тёхї $r\ge R_0$ ш эхъюЄюЁюь $R_0\ge 1,$ уфх
$q_{\alpha, 0}(r)$ -- ёЁхфэхх чэрўхэшх ЇєэъЎшш $Q^{\alpha}(x)$ яю
ёЇхЁх $S(0, r),$ $\alpha>n-1.$}
 \end{corollary}

\medskip
╧юёъюы№ъє юЄюсЁрцхэш  ъырёёр $W_{loc}^{1, n},$ ьхЁр ьэюцхёЄтр Єюўхъ
тхЄтыхэш  ъюЄюЁ√ї Ёртэр эєы■, фы  ъюЄюЁ√ї $K_O(x, f)\in
L_{loc}^{\alpha},$ $\alpha>n-1,$  ты ■Єё  юЄъЁ√Є√ьш ш фшёъЁхЄэ√ьш
(ёь. \cite{MV$_1$}--\cite{MV$_2$}), р Єръцх -- юЄюсЁрцхэш ьш ё
ъюэхўэ√ь шёърцхэшхь фышэ√ (ёь. \cite[ЄхюЁхьр~1 ш
ёыхфёЄтшх~1]{S$_2$}), шьххь Єръцх ёыхфє■∙хх трцэхщ°хх

\medskip
\begin{corollary}\label{cor2}
{\sl ╟ръы■ўхэшх ЄхюЁхь√ \ref{th3} юёЄр╕Єё  т√яюыэхээ√ь, хёыш т
єёыютш ї ¤Єющ ЄхюЁхь√ $f$  ты хЄё  юЄюсЁрцхэшхь ъырёёр $W_{loc}^{1,
n},$ ьхЁр ьэюцхёЄтр Єюўхъ тхЄтыхэш  ъюЄюЁюую Ёртэр эєы■.}
 \end{corollary}

\medskip
┬ ўрёЄэюёЄш, шч ЄхюЁхь√ \ref{th3} ёыхфєхЄ ъырёёшўхёъшщ Ёхчєы№ЄрЄ
╨шъьрэр--┬єюЁшэхэр ю яюыюцшЄхы№эюёЄш ш ЁртэюьхЁэющ юуЁрэшўхээюёЄш
ёэшчє эшцэхую яюЁ фър юЄюсЁрцхэшщ ё юуЁрэшўхээ√ь шёърцхэшхь
\cite{RV}.

\medskip
\section{╧ЁхфтрЁшЄхы№э√х ётхфхэш } ─ы  сюЁхыхтёъюую ьэюцхёЄтр
$A\subset {\Bbb R}^n$ юяЁхфхышь {\it ЇєэъЎш■ ъЁрЄэюёЄш $N(y,f,A)$}
ъръ ўшёыю яЁююсЁрчют Єюўъш $y$ тю ьэюцхёЄтх $A,$ Є.х.
%
$$N(y,f,A)\,=\,{\rm card}\,\left\{x\in E: f(x)=y\right\}\,,\quad
%
N(f,E)\,=\,\sup\limits_{y\in{\Bbb R}^n}\,N(y,f,E)\,.$$
%
╟рьхЄшь, ўЄю ЇєэъЎш  $N(y,f,A)$  ты хЄё  шчьхЁшьющ яю ╦хсхує (ёь.,
эряЁ., \cite[ЄхюЁхьр~IV.1.2]{RR}). ╤ыхфє■∙шщ Ёхчєы№ЄрЄ яЁхфёЄрты хЄ
ёюсющ эхчэрўшЄхы№эюх єёшыхэшх юфэюую шч ъырёёшўхёъшї ьюфєы№э√ї
эхЁртхэёЄт фы  юЄюсЁрцхэшщ ё юуЁрэшўхээ√ь шёърцхэшхь (ёь.
\cite[ЄхюЁхьр~2.4, уы.~II]{Ri}).

\medskip
\begin{theorem}\label{th1}
{\sl ╧єёЄ№ $f:D\rightarrow {\Bbb R}^n$ -- юЄюсЁрцхэшх ё ъюэхўэ√ь
шёърцхэшхь фышэ√, $Q:D\rightarrow [1, \infty]$ -- эхъюЄюЁр 
шчьхЁшьр  яю ╦хсхує ЇєэъЎш , Єрър  ўЄю $K_O(x, f)\le Q(x)$ яюўЄш
тё■фє, $\Gamma$ -- ЇшъёшЁютрээюх ёхьхщёЄтю ъЁшт√ї т $D$ ш $\rho\in
{\rm adm\,}f(\Gamma).$ ╥юуфр
$$M_{1/Q}(\Gamma)\le \int\limits_{{\Bbb R}^n}\rho^n(y)N(y, f, A)dm(y)\,.$$}
\end{theorem}
\begin{proof}
╧єёЄ№ $\rho^{\,\prime}\in {\rm adm\,}f(\Gamma),$ Єюуфр юяЁхфхышь
ЇєэъЎш■ $\rho$ яюырур  $\rho(x)=\rho^{\,\prime}(f(x))\Vert
f^{\,\prime}(x)\Vert$ яЁш $x\in A$ ш $\rho(x)=0$ яЁш $x\in {\Bbb
R}^n\setminus A.$ ╧єёЄ№ $\Gamma_0$ -- ёхьхщёЄтю тёхї ыюъры№эю
ёяЁ ьы хь√ї ъЁшт√ї ёхьхщёЄтр $\Gamma,$ эр ъюЄюЁ√ї $f$ ыюъры№эю
рсёюы■Єэю эхяЁхЁ√тэю, Єюуфр єўшЄ√тр  юяЁхфхыхэшх юЄюсЁрцхэшщ ё
ъюэхўэ√ь шёърцхэшхь фышэ√ ш чрьхўрэшх \ref{rem1}, ь√ яюыєўшь:
$M(\Gamma)=M(\Gamma_0).$ (╬Єё■фр, т ўрёЄэюёЄш, т√ЄхърхЄ, ўЄю Єръцх
$M_{1/Q}(\Gamma)=M_{1/Q}(\Gamma_0)$). ┬ Єръюь ёыєўрх, тшфє
\cite[ыхььр~2.2, уы.~II]{Ri},
$\int\limits_{\gamma} \rho(x) |dx|=\int\limits_{\gamma}
\rho^{\,\prime}(f(x))\Vert f^{\,\prime}(x)\Vert |dx|\ge
\int\limits_{f\circ\gamma} \rho^{\,\prime}(y) |dy| \ge 1,$
ёыхфютрЄхы№эю, $\rho\in {\rm adm\,}\Gamma_0.$ ╙ўшЄ√тр  фы 
яЁюшчтюы№эюую юЄюсЁрцхэш  ё ъюэхўэ√ь шёърцхэшхь фышэ√ тючьюцэюёЄ№
чрьхэ√ яхЁхьхээющ (ёь., эряЁ., \cite[яЁхфыюцхэшх~8.3]{MRSY}, ёь.
Єръцх \cite[яЁхфыюцхэшх~3.7]{MRSY$_1$}), ь√ сєфхь шьхЄ№:
$$M_{1/Q}(\Gamma)=M_{1/Q}(\Gamma_0)\le \int\limits_{{\Bbb R}^n}
\frac{\rho^n(x)}{Q(x)}\,dm(x)=\int\limits_{A}
\frac{\rho^{\,\prime\,n}(f(x))\Vert f^{\,\prime}(x)\Vert^n}{Q(x)}\,
dm(x)\le$$
$$
\le\int\limits_{A} \rho^{\,\prime\,n}(f(x))|J(x, f)|\,
dm(x)=\int\limits_{{\Bbb R}^n}\rho^n(y)N(y, f, A)dm(y)\,.$$ ╥хюЁхьр
фюърчрэр.
\end{proof}$\Box$

\medskip
╧єёЄ№ $S(a, r)$ - яЁюшчтюы№эр  ёЇхЁр, Єюуфр фы  ўшёыр $p\in {\Bbb
R},$ $p\ge 1,$ ш ёхьхщёЄтр ъЁшт√ї $\Gamma,$ шьх■∙шї эюёшЄхы№,
ыхцр∙шщ эр ёЇхЁх $S(a, r),$ юяЁхфхышь ьюфєы№ ёхьхщёЄтр ъЁшт√ї
яюЁ фър $p$ юЄэюёшЄхы№эю $S(a, r)$ ЁртхэёЄтюь
$M_p^S(\Gamma)=\inf_{\rho \in \,{\rm adm}\,\Gamma} \int\limits_{S(a,
r)} \rho^p(x)\,d{\cal H}^{n-1}(x),$
уфх, ъръ юс√ўэю, ${\cal H}^{n-1}$ -- $(n-1)$-ьхЁэр  ьхЁр ╒рєёфюЁЇр.
╒юЁю°ю шчтхёЄэю (ёь. \cite[ЄхюЁхьр~3, ё.~514]{Car}), ўЄю, ъръют√ с√
эх с√ыш ьэюцхёЄтр $E_1, E_2\subset S(a, r),$ $E_1\cap
E_2=\varnothing,$ яЁш $n-1<p<n$ ш $p>n$ тёхуфр
\begin{equation}\label{eq1}
M_p^S(\Gamma(E_1, E_2, S(a, r)))\ge\frac{C_{p, n}}{r^{p+1-n}}\,,
\end{equation}
уфх $C_{p, n}$ -- эхъюЄюЁр  яюёЄю ээр , чртшё ∙р  юЄ ЁрчьхЁэюёЄш
яЁюёЄЁрэёЄтр $n$ ш т√сЁрээюую ўшёыр $p.$ ╤яЁртхфышт ёыхфє■∙шщ
Ёхчєы№ЄрЄ.

\medskip
\begin{theorem}\label{th2}
{\sl ╧єёЄ№ $f:D\rightarrow {\Bbb R}^n$ -- юЄюсЁрцхэшх ё ъюэхўэ√ь
шёърцхэшхь фышэ√ ш $0<c_1<c_2<\infty$ -- эхъюЄюЁ√х ЇшъёшЁютрээ√х
ўшёыр. ╧єёЄ№ Єръцх $Q:D\rightarrow [1, \infty]$ -- шчьхЁшьр  яю
╦хсхує ЇєэъЎш  Єрър , ўЄю $K_O(x, f)\le Q(x)$ яЁш яюўЄш тёхї $x\in
D,$ яЁш ¤Єюь, $Q\in L_{loc}^{\alpha}(D)$ яЁш эхъюЄюЁюь ЇшъёшЁютрээюь
$\alpha>n-1.$ ╧Ёхфяюыюцшь, ўЄю $E$ ш $F$ -- ярЁр ьэюцхёЄт,
єфютыхЄтюЁ ■∙шї єёыютш■ $E\cap S(a, r)\ne \varnothing\ne F\cap S(a,
r)$ яЁш тёхї $r\in (c_1, c_2).$ ╬сючэрўшь $\Delta:=\Gamma(E, F, B(a,
c_2)\setminus \overline{B(a, c_1)}).$ ╥юуфр
\begin{equation}\label{eq27} M_{1/Q}(\Delta)\ge \frac{C_{n,
\alpha}}{\frac{c_2^n}{c_2^n-c_1^n}\cdot\left(\frac{1}{c_2^n-c_1^n}\int\limits_{B(a,
c_2)\setminus B(a, c_1)} Q^{\alpha}(x)\,dm(x)\right)^{1/\alpha}}\,,
\end{equation}
уфх $C_{n, \alpha}$ -- эхъюЄюЁр  яюёЄю ээр , чртшё ∙р  Єюы№ъю юЄ $n$
ш $\alpha.$}
\end{theorem}

\medskip
\begin{proof}
┬√схЁхь $\rho\in {\rm adm}\,\Gamma(E, F, B(a, c_2)\setminus
\overline{B(a, c_1)})$ ═х юуЁрэшўштр  юс∙эюёЄш, ьюцэю ёўшЄрЄ№, ўЄю
%
$\int\limits_D \frac{\rho^n(x)}{Q(x)}\,dm(x)<\infty$ ш
%
$E\cap F=\varnothing.$ ╧єёЄ№ $\alpha$ -- яюыюцшЄхы№эюх ўшёыю шч
єёыютш  ЄхюЁхь√, Єюуфр яюырурхь $p:=\frac{\alpha n}{1+\alpha}.$
╟рьхЄшь, ўЄю, т ¤Єюь ёыєўрх, ўшёыю $p$ єфютыхЄтюЁ хЄ єёыютш■: $p\in
(n-1, n).$ ┬тшфє ёююЄэю°хэш  (\ref{eq1}) яюыєўрхь, ўЄю
\begin{equation}\label{eq2a}
\int\limits_{S(a, r)}\rho^p(x)\,d{\cal H}^{n-1}(x)\ge
M_p^S(\Gamma^r)\ge\frac{C_{p, n}}{r^{p+1-n}}\,,
\end{equation}
уфх $\Gamma^r=\{\gamma\in \Gamma(E, F, B(a, c_2)\setminus
\overline{B(a, c_1)}): |\gamma|\in S(a, r)\}.$ ╚ч (\ref{eq2a})
яюыєўрхь эхЁртхэёЄтю
\begin{equation}\label{eq3}
1\le \widetilde{C_{p, n}}\cdot r\left(r^{1-n}\int\limits_{S(a,
r)}\rho^p(x)\,d{\cal H}^{n-1}(x)\right)^{1/p}\,.
\end{equation}
╧юърцхь ЄхяхЁ№, ўЄю яЁш эхъюЄюЁюь $r\in (c_1, c_2)$ ш $K>n$
т√яюыэхэю эхЁртхэёЄтю
\begin{equation}\label{eq4}
\left(r^{1-n}\int\limits_{S(a, r)}\rho^p(x)\,d{\cal
H}^{n-1}(x)\right)^{1/p}\le
K\cdot\left(\frac{1}{c_2^{n}-c_1^n}\int\limits_{B(a, c_2)\setminus
B(a, c_1)}\rho^p(x)dm(x)\right)^{1/p}\,.\end{equation}
─хщёЄтшЄхы№эю, хёыш $\int\limits_{B(a, c_2)\setminus \overline{B(a,
c_1)}}\rho^p(x)dm(x)=\infty$ фюърч√трЄ№ эхўхую, р хёыш
$$\int\limits_{B(a, c_2)\setminus \overline{B(a,
c_1)}}\rho^p(x)dm(x)=0\,,$$ Єю ш $\int\limits_{S(a,
r)}\rho^p(x)\,d{\cal H}^{n-1}(x)=0$ эр яюўЄш тёхї ёЇхЁрї. ╧єёЄ№
ЄхяхЁ№ $$0<\int\limits_{B(a, c_2)\setminus \overline{B(a,
c_1)}}\rho^p(x)dm(x)<\infty\,.$$ ╧Ёхфяюыюцшь яЁюЄштэюх, р шьхээю,
ўЄю ёююЄэю°хэшх (\ref{eq4}) эрЁє°хэю яЁш тёхї $r\in (c_1, c_2),$
Єюуфр Єръцх
\begin{equation}\label{eq5}
r^{1-n}\int\limits_{S(a, r)}\rho^p(x)\,d{\cal H}^{n-1}(x)\ge
K^p\cdot \frac{1}{(c_2^{n}-c_1^n)}\int\limits_{B(a, c_2)\setminus
B(a, c_1)}\rho^p(x)\,dm(x)\end{equation}
яЁш тёхї $r\in (c_1, c_2).$ ╚эЄхуЁшЁє  ёююЄэю°хэшх (\ref{eq5}) яю
$r\in (c_1, c_2),$ ттшфє ЄхюЁхь√ ╘єсшэш ь√ яЁшїюфшь ъ яЁюЄштюЁхўш■,
ъюЄюЁюх ш фюърч√трхЄ ёююЄэю°хэшх (\ref{eq4}). ─рыхх, яЁшьхэ  
эхЁртхэёЄтю ├╕ы№фхЁр яЁш $p:=\frac{\alpha n}{1+\alpha}$ ш яюырур 
$A:=\frac{1}{c_2^n-c_1^n},$ яюыєўрхь:
$$\left(A\int\limits_{B(a, c_2)\setminus
B(a, c_1)}\rho^p(x) dm(x)\right)^{1/p}\le$$
%
$$\le\left(A\int\limits_{B(a, c_2)\setminus B(a,
c_1)}\frac{\rho^n(x)}{Q(x)}\,dm(x)\right)^{\frac{1}{n}}\cdot
\left(A\int\limits_{B(a, c_2)\setminus B(a, c_1)}Q^{\alpha}(x)\,
dm(x)\right)^{\frac{n-p}{pn}}\,.$$
%
╬с·хфшэ   ёююЄэю°хэш  (\ref{eq3}) ш (\ref{eq4}) тьхёЄх ё яюёыхфэшь
ёююЄэю°хэшхь ш єўшЄ√тр , ўЄю т ёфхырээ√ї яЁхфяюыюцхэш ї $r\le c_2,$
ь√ яюыєўшь ёююЄэю°хэшх (\ref{eq27}) яЁш эхъюЄюЁющ яюёЄю ээющ
$\widetilde{C}_{n, p},$ чртшё ∙хщ Єюы№ъю юЄ $n$ ш $p;$ юфэръю, $p$
ёрью яюыэюёЄ№■ юяЁхфхы хЄё  яю $n$ ш $\alpha,$ Єръ ўЄю т
(\ref{eq27}) ьюцэю чрьхэшЄ№ $\widetilde{C}_{n, p}$ эр $C_{n,
\alpha},$ ўЄю ш ЄЁхсютрыюё№ єёЄрэютшЄ№.
\end{proof} $\Box$

\medskip
╤ыхфє■∙хх юяЁхфхыхэшх шуЁрхЄ трцэє■ Ёюы№ яЁш шёёыхфютрэшш
юЄюсЁрцхэшщ (ёь. \cite[Ёрчф.~3, уы.~II]{Ri}). ╧єёЄ№ $x_1,\ldots,x_k$
-- $k$ Ёрчышўэ√ї Єюўхъ ьэюцхёЄтр $f^{-1}\left(\beta(a)\right)$ ш
$ \widetilde{m} = \sum\limits_{i=1}^k i(x_i,\,f).$
╩Ёштр  $\alpha:[a, c]\rightarrow {\Bbb R}^n$ сєфхЄ эрч√трЄ№ё 
(ўрёЄшўэ√ь) {\it яюфэ Єшхь} ъЁштющ $\beta:[a, b]\rightarrow {\Bbb
R}^n$ ё эрўрыюь т Єюўъх $x,$ $c\le b,$ хёыш $\alpha(a)=x$ ш
$f\circ\alpha(t)=\beta(t)$ яЁш $t\in [a, c.]$ ╧Ёштхф╕ээюх
юяЁхфхыхэшх ЁрёяЁюёЄЁрэ хЄё  Єръцх эр ъЁшт√х, чрфрээ√х эр ъръюь-ышсю
яюыєюЄъЁ√Єюь шэЄхЁтрых.

\medskip
╧єёЄ№ $f:D \rightarrow {\Bbb R}^n$, $n\ge 2,$ -- юЄюсЁрцхэшх,
$\beta: [a,\,b)\rightarrow {\Bbb R}^n$ -- эхъюЄюЁр  ъЁштр  ш
$x\in\,f^{\,-1}\left(\beta(a)\right).$ ╩Ёштр  $\alpha:
[a,\,c)\rightarrow D$ эрч√трхЄё  {\it ьръёшьры№э√ь яюфэ Єшхь} ъЁштющ
$\beta$ яЁш юЄюсЁрцхэшш $f$ ё эрўрыюь т Єюўъх $x,$ хёыш $(1)\quad
\alpha(a)=x;$ $(2)\quad f\circ\alpha=\beta|_{[a,\,c)};$ $(3)$\quad
хёыш $c<c^{\prime}\le b,$ Єю эх ёє∙хёЄтєхЄ ъЁштющ $\alpha^{\prime}:
[a,\,c^{\prime})\rightarrow D,$ Єръющ ўЄю
$\alpha=\alpha^{\prime}|_{[a,\,c)}$ ш $f\circ
\alpha=\beta|_{[a,\,c^{\prime})}.$ ╧юёыхфютрЄхы№эюёЄ№ ъЁшт√ї
$\alpha_1,\dots,\alpha_{\widetilde{m}}$ сєфхЄ эрч√трЄ№ё  {\it
ьръёшьры№эющ яюёыхфютрЄхы№эюёЄ№■ яюфэ Єшщ ъЁштющ $\beta$ яЁш
юЄюсЁрцхэшш $f$ ё эрўрыюь т Єюўърї $x_1,\ldots,x_k,$} хёыш
$(a)$\quad ърцфр  ъЁштр  $\alpha_j$  ты хЄё  ьръёшьры№э√ь яюфэ Єшхь
ъЁштющ $\beta$ яЁш юЄюсЁрцхэшш $f,$ $(b)\quad {\rm
card}\,\left\{j:a_j(a)=x_i\right\}= i(x_i,\,f),\quad 1\le i\le k\,,$
$(c)\quad {\rm card}\,\left\{j:a_j(t)=x\right\}\le i(x,\,f)$ яЁш
тёхї $x\in D$ ш тёхї $t\in I_j,$ уфх $I_j$ -- юсырёЄ№ юяЁхфхыхэш 
ъЁштющ $\alpha_j.$ ╩Ёшт√х $\alpha_1,\dots,\alpha_{\widetilde{m}}$
сєфхь эрч√трЄ№ {\it ёє∙хёЄтхээю юЄфхышь√ьш}, хёыш шч єърчрээ√ї т√°х
єёыютшщ т√яюыэхэю Єюы№ъю юфэю єёыютшх $(c).$ ╬ЄьхЄшь, ўЄю
яюёыхфютрЄхы№эюёЄш ёє∙хёЄтхээю юЄфхышь√ї ьръёшьры№э√ї яюфэ Єшщ ё
эрўрыюь т ЇшъёшЁютрээ√ї Єюўърї (т ўрёЄэюёЄш -- ъръюх-ышсю юфэю
ьръёшьры№эюх яюфэ Єшх) яЁш юЄъЁ√Є√ї фшёъЁхЄэ√ї юЄюсЁрцхэш ї тёхуфр
ёє∙хёЄтє■Є (ёь. \cite[ЄхюЁхьр~3.2, уы.~II]{Ri}).

\medskip
╟фхё№ ш фрыхх $l\left(f^{\,\prime}(x)\right)=\min\limits_{h\in {\Bbb
R}^n \setminus \{0\}} \frac {|f^{\,\prime}(x)h|}{|h|}.$ {\it
┬эєЄЁхээ   фшырЄрЎш } $K_I(x, f)$ юЄюсЁрцхэш  $f$ т Єюўъх $x$
юяЁхфхы хЄё  ЁртхэёЄтюь
$$K_{I}(x,f)\quad =\quad\left\{
\begin{array}{rr}
\frac{|J(x,f)|}{{l\left(f^{\,\prime}(x)\right)}^n}, & J(x,f)\ne 0,\\
1,  &  f^{\,\prime}(x)=0, \\
\infty, & \text{т\,\,юёЄры№э√ї\,\,ёыєўр ї}
\end{array}
\right.\,. $$
╩ръ шчтхёЄэю,
\begin{equation}\label{eq9C}
K_I(x,f)\le K_O^{n-1}(x,f),\qquad K_O(x,f)\le K_I^{n-1}(x,f)
\end{equation}
(ёь. \cite[ёююЄэю°хэш ~(2.7) ш (2.8), я.~2.1, уы.~I]{Re}), ш ўЄю
$K_I(x,f)\ge 1$ ш $K_O(x,f)\ge 1$ тё■фє, уфх ¤Єш тхышўшэ√
юяЁхфхыхэ√. ═ряюьэшь х∙╕ юфшэ эхюсїюфшь√щ эрь Ёхчєы№ЄрЄ, фюърчрээ√щ
Ёрэхх ртЄюЁюь (ёь., эряЁ., \cite[ЄхюЁхьр~3.1]{Sev}).

\medskip
\begin{proposition}\label{pr1}{\sl\,
╧єёЄ№ $f:D\rightarrow {\Bbb R}^n$ -- юЄъЁ√Єюх фшёъЁхЄэюх юЄюсЁрцхэшх
ё ъюэхўэ√ь шёърцхэшхь фышэ√, $\Gamma$ -- ёхьхщёЄтю ъЁшт√ї т $D,$
$\Gamma^{\,\prime}$ -- ёхьхщёЄтю ъЁшт√ї т ${\Bbb R}^n$ ш $m$ --
эрЄєЁры№эюх ўшёыю, Єръюх ўЄю т√яюыэхэю ёыхфє■∙хх єёыютшх. ─ы  ърцфющ
ъЁштющ $\beta\in \Gamma^{\,\prime}$ эрщфєЄё  ъЁшт√х
$\alpha_1,\ldots,\alpha_m$ ёхьхщёЄтр $\Gamma$ Єръшх ўЄю $f\circ
\alpha_j\subset \beta$ фы  тёхї $j$ ш ЁртхэёЄтю $\alpha_j(t)=x$
шьххЄ ьхёЄю яЁш тёхї $x\in D,$ тёхї $t$ ш эх сюыхх ўхь $i(x,f)$
шэфхъёрї $j.$ ╥юуфр
$ M(\Gamma^{\,\prime} )\le\frac{1}{m}\quad\int\limits_D
K_I(x,\,f)\cdot \rho^n (x)\,dm(x)
$
фы  ърцфющ ЇєэъЎшш $\rho \in {\rm adm}\,\Gamma.$}
\end{proposition}

\medskip ╤юуырёэю \cite[юяЁхфхыхэшх~6.6]{Va$_1$}, ёхьхщёЄтр ъЁшт√ї
$\Gamma_1, \Gamma_2,\ldots$ сєфхь эрч√трЄ№ {\it юЄфхышь√ьш}, хёыш
эрщфєЄё  эхяхЁхёхър■∙шхё  сюЁхыхтёъшх ьэюцхёЄтр $E_i\subset {\Bbb
R}^n$ Єръшх, ўЄю $\int\limits_{\gamma}g_i(x) |dx|=0$ фы  ърцфющ
ыюъры№эю ёяЁ ьы хьющ ъЁштющ $\gamma\in \Gamma_i,$ уфх $g_i$ --
їрЁръЄхЁшёЄшўхёър  ЇєэъЎш  ьэюцхёЄтр ${\Bbb R}^n\setminus E_i.$
─юърцхь ёыхфє■∙хх тёяюьюурЄхы№эюх єЄтхЁцфхэшх (ёь. Єръцх
\cite[ыхььр~1.3, уы.~IV]{Ri}).

\medskip
\begin{lemma}\label{lem1}
{\sl ╧єёЄ№ $f:D\rightarrow {\Bbb R}^n$ -- юЄъЁ√Єюх фшёъЁхЄэюх
юЄюсЁрцхэшх ё ъюэхўэ√ь шёърцхэшхь фышэ√, $m:{\Bbb
S}^{n-1}\rightarrow {\Bbb Z}$ -- эхъюЄюЁр  эхюЄЁшЎрЄхы№эр 
Ўхыюўшёыхээр  сюЁхыхтёър  ЇєэъЎш  ш фы  ърцфюую $y\in {\Bbb
S}^{n-1}$ чряшё№ тшфр $\beta_y:[s, t]\rightarrow {\Bbb R}^n$
юсючэрўрхЄ ъЁштє■ $\beta_y(u)=yu.$ ╬сючэрўшь ўхЁхч $\Gamma^{*}$
ёхьхщёЄтю ъЁшт√ї, ёюёЄю ∙шї шч $m(y)$ ўрёЄшўэ√ї ёє∙хёЄтхээю
юЄфхышь√ї яюфэ Єшщ ъЁштющ $\beta_y$ яЁш юЄюсЁрцхэшш $f$
(яЁхфяюыюцшь, ўЄю тёх ¤Єш $m(y)$ яюфэ Єшщ ёє∙хёЄтє■Є). ╥юуфр фы 
ърцфющ ЇєэъЎшш $\rho\in {\rm adm}\,\Gamma^*$ т√яюыэ хЄё  эхЁртхэёЄтю
$\int\limits_{{\Bbb S}^{n-1}}m(y)d{\cal H}^{n-1}(y)\le
\left(\log\frac{t}{s}\right)^{n-1} \int\limits_{D}K_I(x,
f)\cdot\rho^n(x)dm(x).$}
\end{lemma}

\begin{proof}
╧юырурхь $E_k:=\{y\in {\Bbb S}^{n-1}: m(y)=k\},$ $k\in {\Bbb N}\cup
\{0\},$ $\Gamma_k:=\{\beta_y: y\in E_k\}.$ ╧єёЄ№ Єръцх
$\Gamma_k^{\,*}$ хёЄ№ яюфёхьхщёЄтю ъЁшт√ї $\Gamma^{\,*},$ фы 
ъюЄюЁ√ї ёююЄтхЄёЄтє■∙р  ъЁштр  $\beta_y\in \Gamma_k.$ ┬тшфє
яЁхфыюцхэш  \ref{pr1}
\begin{equation}\label{eq8}
k\cdot M(\Gamma_k)\quad\le\quad \int\limits_D K_I(x,\,f)\cdot \rho^n
(x)\,dm(x)
\end{equation}
фы  ърцфющ ЇєэъЎшш $\rho \in {\rm adm}\,\Gamma_k^{\,*}.$ ┬тшфє
\cite[Ёрчф.~7.7]{Va$_1$},
\begin{equation}\label{eq9}
M(\Gamma_k)=\frac{{\cal
H}^{n-1}(E_k)}{\left(\log(t/s)\right)^{n-1}}\,.
\end{equation}
╟рьхЄшь, ўЄю ёхьхщёЄтр $\Gamma_k$ юЄфхышь√; Єюуфр $M_{K_I(\cdot,
f)}(\Gamma^{\,*})=\sum\limits_{k=0}^{\infty} M_{K_I(\cdot,
f)}(\Gamma_k^{\,*})$ (ёь. \cite[яєэъЄ (b), ё.~176  ш яєэъЄ (e),
ё.~178]{Fu}, ёь. Єръцх \cite[ЄхюЁхьр~15.1, уы.~I]{Sa}). ╥юуфр
ёєььшЁє  яю $k$ ёююЄэю°хэшх (\ref{eq8}) ш єўшЄ√тр  яЁш ¤Єюь
ёююЄэю°хэшх (\ref{eq9}), яюыєўрхь єЄтхЁцфхэшх ыхьь√.
\end{proof}$\Box$

\medskip
─ы  яЁюшчтюы№эюую фшёъЁхЄэюую юЄюсЁрцхэш  $f:D\rightarrow {\Bbb
R}^n$ ш сюЁхыхтёъюую ьэюцхёЄтр $E\subset D$ юяЁхфхышь Єръ эрч√трхьє■
{\it ёўшЄр■∙є■ ЇєэъЎш■} яю ёыхфє■∙хьє яЁртшыє:
$$n(E, y)=\sum\limits_{x\in f^{\,-1}(y)\cap E}i(x, f)\,,$$
уфх $i(x, f)$ -- ыюъры№э√щ Єюяюыюушўхёъшщ шэфхъё (cь. \cite{RR}).
╬яЁхфхышь Єръцх {\it шэЄхуЁры№эюх ёЁхфэхх} ёўшЄр■∙хщ ЇєэъЎшш $n(E,
\cdot)$ яю ёЇхЁх $S(y, t)$ ёююЄэю°хэшхь
$$\nu(E, y, t):=\frac{1}{\omega_{n-1}}\int\limits_{{\Bbb S}^{n-1}}n(E, y+tx)
\,d{\cal H}^{n-1}(x)\,.$$ ╧Ёш $E=B(a, r)$ яюырурхь $\nu(a, r, y,
t):=\nu(E, y, t).$

\medskip
─ы  Єюўъш $x_0\in D$ ш ўшёхы $0<r_1<r_2<\infty$ яюырурхь
$A(r_1, r_2, x_0)=\{x\in D: r_1<|x-x_0|<r_2\}.$ └эрыюу ёыхфє■∙хую
Ёхчєы№ЄрЄр ьюцхЄ с√Є№ эрщфхэ т \cite[ыхььр~1.1, уы.~IV]{Ri} (ёь.
Єръцх т \cite[ыхььр~2.6]{Ra$_1$}).

\medskip
\begin{lemma}\label{lem3}
{\sl ╧єёЄ№ $f:D\rightarrow {\Bbb R}^n$ -- юЄъЁ√Єюх фшёъЁхЄэюх
юЄюсЁрцхэшх ё ъюэхўэ√ь шёърцхэшхь фышэ√, $\theta>1,$ $\overline{B(a,
\theta r)}\subset D,$ $y\in {\Bbb R}^n$ ш $s, t>0.$ ╥юуфр
$$\nu(a, \theta r, y, s)\ge \nu(a, r, y, t)- |\log(t/s)|^{n-1}\cdot \frac{1}
{\left(\int\limits_{r}^{\theta r}\ \frac{d\omega}{\omega
k_{a}^{\frac{1}{n-1}}(\omega)}\right)^{n-1}} \,,$$
уфх $k_{a}$ юяЁхфхыхэю ёююЄэю°хэшхь (\ref{eq32*}) яЁш $Q:=K_I(x, f)$
ш $x_0:=a.$}
\end{lemma}

\medskip
\begin{proof}
═х юуЁрэшўштр  юс∙эюёЄш Ёрёёєцфхэшщ, ьюцэю ёўшЄрЄ№, ўЄю $s<t$
(фюърчрЄхы№ёЄтю т ёыєўрх $t>s$ яЁютюфшЄё  рэрыюушўэю). ─ы 
ЇшъёшЁютрээюую $z\in {\Bbb S}^{n-1}$ яюырурхь
$m(z):=\max\{0, n(B(a, r), y+sz)- n(B(a, \theta r), y+tz)\}.$
╧Ёхфяюыюцшь, ўЄю $m(z)>0.$ ┬тшфє \cite[ЄхюЁхьр~3.2, уы.~II]{Ri}
эрщфєЄё  $n(B(a, r), y+sz)$ ьръёшьры№э√ї яюфэ Єшщ ъЁштющ
$\beta_z=zu,$ $u\in (s, t),$ ё эрўрыюь т °рЁх $\overline{B(a, r)}$ ш
ыхцр∙шї т °рЁх $B(a, \theta r).$ ┬тшфє Єющ цх ЄхюЁхь√, тёх ¤Єш
яюфэ Єш  ёє∙хёЄтхээю юЄфхышь√ (ёь. Єрь цх).

\medskip
╨рёёьюЄЁшь ъръюх-ышсю шч Єръшї яюфэ Єшщ $\alpha:[s, c]\rightarrow
B(a, \theta r).$ ┼ёыш яЁхфхы№эюх ьэюцхёЄтю ъЁштющ $\alpha$ Ўхышъюь
ыхцшЄ тэєЄЁш °рЁр $B(a, \theta r),$ Єю, тю-яхЁт√ї, т ёшыє
фшёъЁхЄэюёЄш юЄюсЁрцхэш  $f$ юэю  ты хЄё  юфэюЄюўхўэ√ь; тю-тЄюЁ√ї, т
¤Єюь ёыєўрх $f(\alpha(c))=zt.$ ╟рьхЄшь, ўЄю юс∙хх ъюышўхёЄтю ъЁшт√ї
ёЁхфш $n(B(a, r), y+sz)$ ьръёшьры№э√ї яюфэ Єшщ ъЁштющ $\beta_z,$
ёюфхЁцр∙шї Єръє■ Єюўъє $x:=\alpha(c),$ эх яЁхт√°рхЄ ўшёыр $k:={\rm
card}\{f^{\,-1}(zt)\cap B(a, \theta r)\}\le n(B(a, \theta r),
y+tz).$ ╥ръшь юсЁрчюь, ёЁхфш $n(B(a, r), y+sz)$ яюфэ Єшщ $\alpha,$
яю ъЁрщэхщ ьхЁх, $n(B(a, r), y+sz)-k\ge m(z)$ ъЁшт√ї эх ёюфхЁцрЄ
яЁююсЁрчют Єюўъш $zt$ яЁш юЄюсЁрцхэшш $f$ т °рЁх $B(a, \theta r).$
╥ръшь юсЁрчюь, $m(z)$ ъЁшт√ї $\alpha$ Єръют√, ўЄю ${\rm
dist}\,(\alpha(u), S(a, \theta r))\rightarrow 0$ яЁш $u\rightarrow
c-0.$

\medskip
┬тшфє ыхьь√ \ref{lem1} ь√ яюыєўшь, ўЄю яЁш тёхї фюёЄрЄюўэю ьры√ї
$\varepsilon>0$ ш яЁюшчтюы№эющ ЇєэъЎшш $\rho\in {\rm
adm}\,\Gamma(S(a, r), S(a, \theta r-\varepsilon), A(r, \theta
r-\varepsilon, a))$
\begin{equation}\label{eq10}
\int\limits_{{\Bbb S}^{n-1}}m(z)d{\cal H}^{n-1}(z)\le
\left(\log\frac{t}{s}\right)^{n-1} \int\limits_{D} K_I(x,
f)\cdot\rho^n(x)dm(x)\,.
\end{equation}
╧юырурхь
\begin{equation}\label{eq9A}
I=I(x_0,r_1,r_2)=\int\limits_{r_1}^{r_2}\
\frac{dr}{rq_{x_0}^{\frac{1}{n-1}}(r)}\,,
\end{equation}
уфх $q_{x_0}$ юяЁхфхыхэю т (\ref{eq32*}). ╧єёЄ№ $I=I(a, r, \theta
r-\varepsilon)\ne 0,$ уфх $I$ юяЁхфхыхэю т (\ref{eq9A}) яЁш
$Q:=K_I(x, f),$ $x_0:=a,$ $r_1:=r$ ш $r_2=\theta r-\varepsilon.$
(╤ыхфєхЄ юЄьхЄшЄ№, ўЄю $I<\infty,$ яюёъюы№ъє $Q\ge 1$). ╧Ёхфяюыюцшь
тэрўрых, ўЄю $I>0.$ ╧юырурхь
$$\psi(t)=\left \{\begin{array}{rr}
1/[\omega k_{a}^{\frac{1}{n-1}}(\omega)]\ , & \ \omega\in (r ,\theta
r-\varepsilon)\ ,
\\ 0\ ,  &  \ \omega\notin (r, \theta r-\varepsilon)\ ,
\end{array} \right. $$
уфх ЇєэъЎш  $k_{a}(\omega)$ юяЁхфхыхэр яю ёююЄэю°хэш■ (\ref{eq32*})
яЁш $Q:=K_I(x, f).$ ╥юуфр ттшфє ЄхюЁхь√ ╘єсшэш
\begin{equation}\label{eq3A}
\int\limits_{A} Q(x)\cdot\psi^n(|x-a|)dm(x)=\omega_{n-1} I\,,
\end{equation}
уфх $A=A(r ,\theta r-\varepsilon, a).$ ╟рьхЄшь, ўЄю ЇєэъЎш 
$\eta_1(\omega):=\psi(\omega)/I,$ $\omega\in (r, \theta
r-\varepsilon),$ єфютыхЄтюЁ хЄ ёююЄэю°хэш■ $\int\limits_{r}^{\theta
r-\varepsilon}\eta_1(\omega)d \omega=1.$ ╤ фЁєующ ёЄюЁюэ√, юЄ√∙хЄё 
сюЁхыхтёър  ЇєэъЎш  $\eta(\omega)$ Єрър , ўЄю
$\eta_1(\omega)=\eta(\omega)$ яЁш яюўЄш тёхї $\omega\in (\theta,
\theta r-\varepsilon)$ (ёь. \cite[Ёрчф.~2.3.6]{Fe}). ┬ Єръюь ёыєўрх,
ттшфє \cite[ЄхюЁхьр~5.7]{Va$_1$} ЇєэъЎш  $\rho(x):=\eta(|x-a|)\in
{\rm adm}\,\Gamma(S(a, r), S(a, \theta r-\varepsilon), A(r, \theta
r-\varepsilon, a))$ ш, чэрўшЄ, шч (\ref{eq10}) ш (\ref{eq3A})
т√ЄхърхЄ, ўЄю
\begin{equation}\label{eq11A}
\int\limits_{{\Bbb S}^{n-1}}m(z)d{\cal H}^{n-1}(z)\le
\left(\log\frac{t}{s}\right)^{n-1}\cdot \frac{\omega_{n-1}}
{\left(\int\limits_{r}^{\theta r-\varepsilon}\ \frac{d\omega}{\omega
k_{a}^{\frac{1}{n-1}}(\omega)}\right)^{n-1}}\,.
\end{equation}
╟рьхЄшь, ўЄю яюёыхфэхх ёююЄэю°хэшх ёяЁртхфыштю Єръцх ш т ёыєўрх,
ъюуфр $\int\limits_{r}^{\theta r-\varepsilon}\frac{d\omega}{\omega
k_{a}^{\frac{1}{n-1}}(\omega)}=0.$ ╧юёъюы№ъє $Q:=K_I(x, f)\ge 1,$
ЇєэъЎш  $\frac{d\omega}{\omega k_{a}^{\frac{1}{n-1}}(\omega)}$
 ты хЄё  шэЄхуЁшЁєхьющ яю $\omega$ эр $(r, \theta r)$ ш, чэрўшЄ, яю
ЄхюЁхьх юс рсёюы■Єэющ эхяЁхЁ√тэюёЄш шэЄхуЁрыр (ёь.
\cite[ЄхюЁхьр~13.2, уы.~I]{Sa}) $\frac{\omega_{n-1}}
{\int\limits_{r}^{\theta r-\varepsilon}\frac{d\omega}{\omega
k_{a}^{\frac{1}{n-1}}(\omega)}}\rightarrow \frac{\omega_{n-1}}
{\int\limits_{r}^{\theta r}\frac{d\omega}{\omega
k_{a}^{\frac{1}{n-1}}(\omega)}}$ яЁш $\varepsilon\rightarrow 0.$
╥ръшь юсЁрчюь, шч (\ref{eq11A}) т√ЄхърхЄ, ўЄю
\begin{equation}\label{eq12}
\int\limits_{{\Bbb S}^{n-1}}m(z)d{\cal H}^{n-1}(z)\le
\left(\log\frac{t}{s}\right)^{n-1}\cdot \frac{\omega_{n-1}}
{\left(\int\limits_{r}^{\theta r}\ \frac{d\omega}{\omega
k_{a}^{\frac{1}{n-1}}(\omega)}\right)^{n-1}}\,.
\end{equation}
╬сючэрўшь фрыхх ёшьтюыюь $E$ ьэюцхёЄтю $E:=\{y\in {\Bbb S}^{n-1}:
m(y)>0\}.$ ╥юуфр
$$\int\limits_{{\Bbb S}^{n-1}} n(B(a, \theta r), y+tz)\, d{\cal
H}^{n-1}(z)=$$
$$=\int\limits_{{\Bbb S}^{n-1}\setminus E} n(B(a, \theta r), y+tz)\,
d{\cal H}^{n-1}(z)+ \int\limits_{E} n(B(a, \theta r), y+tz)\, d{\cal
H}^{n-1}(z)\ge$$
$$
\ge \int\limits_{{\Bbb S}^{n-1}\setminus E} n(B(a, r), y+sz)\,
d{\cal H}^{n-1}(z)+ \int\limits_{E} n(B(a, r), y+ sz)\, d{\cal
H}^{n-1}(z)-$$
$$-\int\limits_{E} m(z)\, d{\cal H}^{n-1}(z)=$$
\begin{equation}\label{eq14}
=\int\limits_{{\Bbb S}^{n-1}} n(B(a, r), y+sz)\, d{\cal H}^{n-1}(z)
- \int\limits_{{\Bbb S}^{n-1}} m(z)\, d{\cal H}^{n-1}(z)\,.
\end{equation}
╧ю юяЁхфхыхэш■ тхышўшэ√ $\nu$ шч (\ref{eq12}) ш (\ref{eq14})
т√ЄхърхЄ, ўЄю
$\nu(a, \theta r, y, tz)\ge \nu(a, r, y, sz)-
\left(\log\frac{t}{s}\right)^{n-1}\cdot \frac{1}
{\left(\int\limits_{r}^{\theta r}\ \frac{d\omega}{\omega
k_{a}^{\frac{1}{n-1}}(\omega)}\right)^{n-1}},$
ўЄю ш ЄЁхсютрыюё№ єёЄрэютшЄ№.
\end{proof}$\Box$

\medskip
└эрыюу ёыхфє■∙хую єЄтхЁцфхэш  фы  юЄюсЁрцхэшщ ё ъюэхўэ√ь шёърцхэшхь
фюърчрэ т ЁрсюЄх \cite[ыхььр~2.7]{Ra$_1$}.

\medskip
\begin{lemma}\label{lem2}
{\sl ╧єёЄ№ $f:D\rightarrow {\Bbb R}^n$ -- юЄюсЁрцхэшх ё ъюэхўэ√ь
шёърцхэшхь фышэ√, $E$ ш $F$ -- эхяхЁхёхър■∙шхё  ьэюцхёЄтр, ыхцр∙шх т
°рЁх $\overline{B(a, R)}$ Єръшх, ўЄю $f(E)\subset B(z, s),$
$f(F)\subset {\Bbb R}^n\setminus B(z, t),$ $s<t.$ ╧єёЄ№ Єръцх яЁш
эхъюЄюЁюь $\theta>1$ шьххЄ ьхёЄю тъы■ўхэшх $\overline{B(a, \theta
R)}\subset D.$ ╥юуфр, хёыш $\Gamma$ -- ёхьхщёЄтю тёхї ъЁшт√ї,
ёюхфшэ ■∙шї $E$ ш $F$ т °рЁх $B(a, R)$ ш, ъЁюьх Єюую, $K_O(x, f)\le
Q(x)$ яюўЄш тё■фє яЁш эхъюЄюЁющ шчьхЁшьющ ЇєэъЎшш $Q:D\rightarrow
[1, \infty],$ Єю
$$M_{1/Q}(\Gamma)\le \omega_{n-1}\left(
\frac{\nu(a, \theta R, z, t)}{\log^{n-1}(t/s)} +
\frac{1}{\left(\int\limits_{R}^{\theta R} \frac{d\omega}{\omega
\widetilde{q}^{\frac{1}{n-1}}_{a}(\omega)}\right)^{n-1}}\right)\,,$$}
уфх $\widetilde{q}_{a}(\omega)$ -- ёЁхфэхх шэЄхуЁры№эюх чэрўхэшх
ЇєэъЎшш $Q^{n-1}$ эр ёЇхЁх $S(a, \omega).$
\end{lemma}

\medskip
\begin{proof} ═х юуЁрэшўштр  юс∙эюёЄш Ёрёёєцфхэшщ, ьюцэю ёўшЄрЄ№, ўЄю $z=0.$
┬√схЁхь $\rho\in {\rm adm}\, f(\Gamma),$ яюырур 
$$\rho(y)=\left \{\begin{array}{rr} \frac{1}{\log (t/s)|y|}\ , & \
s<|y|<t,
\\ 0\ ,  & \ y\in {\Bbb R}^n\setminus A(s, t, 0).
\end{array} \right.$$
╧ю ЄхюЁхьх \ref{th1}
$$M_{1/Q}(\Gamma)\le \int\limits_{{\Bbb R}^n}\rho^n(y)n(B(a, R),
y)dm(y)=$$
$$=\frac{1}{\log^{n} (t/s)}\cdot\int\limits_{s}^t \frac{1}{r^n}
\int\limits_{S(0, r)} n(B(a, R), y) d{\cal H}^{n-1}(y) dr=$$
$$=\frac{1}{\log^{n} (t/s)}\cdot\int\limits_{s}^t \frac{1}{r}
\int\limits_{{\Bbb S}^{n-1}} n(B(a, R), ry) d{\cal H}^{n-1}(y) dr=$$
\begin{equation}\label{eq15}
=\frac{\omega_{n-1}}{\log^{n} (t/s)}\cdot\int\limits_{s}^t
\frac{1}{r}\cdot \nu(a, R, 0, r)dr\,.
\end{equation}
╧ю ыхььх \ref{lem3} ё єў╕Єюь эхЁртхэёЄт (\ref{eq9C}) фы  тё ъюую
$r\in (s, t)$
$$\nu(a, R, 0, r) \le \nu(a, \theta R, 0, t)+ (\log(t/s))^{n-1}\cdot \frac{1}
{\left(\int\limits_{R}^{\theta R}\ \frac{d\omega}{\omega
\widetilde{q}^{\frac{1}{n-1}}_{a}(\omega)}\right)^{n-1}} \,,$$
уфх $\widetilde{q}_{a}(\omega)$ юсючэрўрхЄ ёЁхфэхх шэЄхуЁры№эюх
чэрўхэшх ЇєэъЎшш $Q^{n-1}$ эр ёЇхЁх $S(a, \omega).$ ╥юуфр шч
ёююЄэю°хэшщ т (\ref{eq15}) ь√ яюыєўрхь эхюсїюфшьюх чръы■ўхэшх.
\end{proof}$\Box$

\medskip
\section{─юърчрЄхы№ёЄтю юёэютэюую Ёхчєы№ЄрЄр} {\it ─юърчрЄхы№ёЄтю
ЄхюЁхь√ \ref{th3}.} ═х юуЁрэшўштр  юс∙эюёЄш Ёрёёєцфхэшщ, ьюцэю
ёўшЄрЄ№, ўЄю $b=0.$ ╧юырурхь $E=\gamma([0, 1)).$ ╧юёъюы№ъє $b=0$ --
рёшьяЄюЄшўхёъшщ яЁхфхы юЄюсЁрцхэш  $f$ т схёъюэхўэюёЄш, эрщф╕Єё 
ўшёыю $R>0$ Єръюх, ўЄю $f(E\cap({\Bbb R}^n\setminus B(0, R)))\subset
{\Bbb B}^n.$ ╨рёёьюЄЁшь ъЁштє■ $\beta(t)=y_0t,$ $t\in [1, \infty),$
уфх $y_0$ Єръютю, ўЄю $|y_0|:=\max\limits_{x\in B(0,
R)}|f(x)|=|f(x_0)|,$ $x_0\in S(0, R)$ ш $y_0\in S(0, M_f(R)).$ ═х
юуЁрэшўштр  юс∙эюёЄш, ьюцэю ёўшЄрЄ№, ўЄю $M_f(R)>1.$ ╧юёъюы№ъє $f$
-- фшёъЁхЄэюх ш юЄъЁ√Єюх юЄюсЁрцхэшх, ёє∙хёЄтєхЄ ьръёшьры№эюх
яюфэ Єшх $\alpha:[1, c)\rightarrow {\Bbb R}^n$ ъЁштющ $\beta$ т
${\Bbb R}^n$ ё эрўрыюь т Єюўъх $x_0$ (ёь. \cite[ЄхюЁхьр~3.2,
уы.~II]{Ri}). ╧єёЄ№ $F_R$ -- ъюьяюэхэЄр ёт чэюёЄш ьэюцхёЄтр
$f^{\,-1}({\Bbb R}^n\setminus \overline{B(0, M_f(R))}),$ ёюфхЁцр∙р 
¤Єє ъЁштє■ $\alpha|_{(0, c)},$ Єюуфр $x_0\in \overline{F_R}\cap S(0,
R)\ne \varnothing$ ш $\alpha(t)\rightarrow \infty$ яЁш $t\rightarrow
c-0.$ (┬ ўрёЄэюёЄш, юЄё■фр ёыхфєхЄ, ўЄю ъюьяюэхэЄр $F_R$
эхюуЁрэшўхэр). ╟рьхЄшь, ўЄю $\left(E\cap({\Bbb R}^n\setminus B(0,
R))\right)\cap F_R=\varnothing.$ ┬тшфє ыхьь√ \ref{lem2} фы 
ёхьхщёЄтр ъЁшт√ї $\Gamma,$ ёюхфшэ ■∙шї ьэюцхёЄтр $E$ ш $F_R$ т
ъюы№Ўх $B(0, \theta R)\setminus B(0, R),$ ь√ сєфхь шьхЄ№, ўЄю яЁш
$K> 1$
\begin{equation}\label{eq37}
\nu(0, K\theta R, 0,
1)\ge\frac{\log^{n-1}M_f(R)}{\omega_{n-1}}\left(M_{1/Q}(\Gamma)-\frac{1}
{\left(\int\limits_{\theta R}^{K\theta R}\frac{d\omega}{\omega
\widetilde{q}^{\frac{1}{n-1}}_{0}(\omega)}\right)^{n-1}}\right)\,,
\end{equation}
уфх $\widetilde{q}_0(\omega)$ -- ёЁхфэхх шэЄхуЁры№эюх чэрўхэшх
ЇєэъЎшш $Q^{n-1}(x)$ эр ёЇхЁх $S(0, \omega).$ ╚ч ЄхюЁхь√ \ref{th2}
т√ЄхърхЄ, ўЄю яЁш $\theta>\theta_1$ ш $R>R_0$
$$M_{1/Q}(\Gamma)\ge \frac{C_{n,
\alpha}}{\frac{\theta^nR^n}{\theta^nR^n-R^n}\left(\cdot\frac{1}{\theta^nR^n-R^n}\int\limits_{B(0,
\theta R)\setminus B(0, R)}
Q^{\alpha}(x)\,dm(x)\right)^{1/\alpha}}\ge $$
$$\ge\frac{C_{n, \alpha}}{\frac{C_0}{\theta^nR^n}\int\limits_{B(0,
\theta R)} Q^{\alpha}(x)\,dm(x)} \,.$$
╥юуфр шч (\ref{eq37}), єўшЄ√тр  єёыютш  (\ref{eq18}) ш (\ref{eq16})
(яЁш $\theta\ge\theta_2$ ш $K\ge K_0,$ уфх $K_0$ -- эхъюЄюЁюх
фюёЄрЄюўэю сюы№°юх ўшёыю), яюыєўрхь
\begin{equation}\label{eq17}
\nu(0, K\theta R, 0, 1)\ge C_1\cdot\log^{n-1}M_f(R)
\end{equation}
яЁш эхъюЄюЁющ яюёЄю ээющ $C_1>0.$ ╧Ёртр  ўрёЄ№ яюёыхфэхую
ёююЄэю°хэш  эшъръ эх чртшёшЄ юЄ ярЁрьхЄЁют $K$ ш $\theta,$ яю¤Єюьє
яЁюшчтхф  яхЁхюсючэрўхэш  т (\ref{eq17}), ь√ сєфхь шьхЄ№
\begin{equation}\label{eq19}
\nu(0, \theta R, 0, 1)\ge C_1\cdot\log^{n-1}M_f(R)
\end{equation}
фы  яЁюшчтюы№эюую $\theta\ge\theta_3>1.$ ╤ фЁєующ ёЄюЁюэ√, яюёъюы№ъє
$f$ -- юЄъЁ√Єюх юЄюсЁрцхэшх, эшърър  Єюўър ёЇхЁ√ $S(0, M_f(R))$ эх
 ты хЄё  юсЁрчюь Єюўъш юЄъЁ√Єюую °рЁр $B(0, R)$ яЁш юЄюсЁрцхэшш $f$
ш, чэрўшЄ,
$$\nu(0, K\theta R, 0, M_f(K\theta R))=0$$ фы  яЁюшчтюы№эюую $K>1.$
╥юуфр (яю ыхььх \ref{lem3}) $0=\nu(0, K\theta R, 0, M_f(K\theta
R))\ge$
$$\ge \nu(0, \theta R, 0, 1)- \log M_f^{n-1}(K\theta R)\cdot \frac{1}
{\left(\int\limits_{\theta R}^{K\theta R}\ \frac{d\omega}{\omega
\widetilde{q}^{\frac{1}{n-1}}_{0}(\omega)}\right)^{n-1}}\,.$$
╬Єё■фр ш шч єёыютш  (\ref{eq16}) т√ЄхърхЄ, ўЄю яЁш тёхї $K\ge K_1,$
$\theta\ge\theta_4>1$ ш $R>R_0$
\begin{equation}\label{eq20}
\nu(0, \theta R, 0, 1)\le C_1/2\cdot \log^{n-1}M_f(K\theta R)\,.
\end{equation}
╚ч (\ref{eq19}) ш (\ref{eq20}) ёыхфєхЄ, ўЄю яЁш
$\theta\ge\max\{\theta_1, \theta_2, \theta_3, \theta_4\},$ $R>R_0$ ш
$K\ge K_2$
\begin{equation}\label{eq21}
\log^{n-1}M_f(K\theta R)\ge 2\cdot \log^{n-1}M_f(R)\,.
\end{equation}
┬√яшё√тр  ёююЄэю°хэшх (\ref{eq21}) фы  яюёыхфютрЄхы№эюёЄш $R_0=R,$
$R_1=K\theta R,$ $\ldots,$ $R_m=K^m\theta^m R, \ldots$ ь√ яюыєўшь:
$$\log^{n-1}M_f(K^m\theta^m  R)\ge \ldots \ge 2^m\cdot
\log^{n-1}M_f(R)\,,$$
юЄъєфр $(n-1)\log\log M_f(K^m\theta^m  R)\ge m\log 2 + (n-1)\log
\log M_f(R)$ ш
$$(n-1)\frac{\log\log M_f(K^m\theta^m  R)}{\log R_m}\ge
\frac{m\log 2}{m\log (K\theta)+\log R}+(n-1)\frac{\log\log
M_f(R)}{m\log (K\theta)+\log R}\,.$$
╬Єё■фр яюър ўЄю ёыхфєхЄ, ўЄю фы  ы■сющ яюфяюёыхфютрЄхы№эюёЄш эюьхЁют
$m_k,$ фы  ъюЄюЁющ яЁхфхы ыхтющ ўрёЄш яюёыхфэхую ёююЄэю°хэш 
ёє∙хёЄтєхЄ, т√яюыэхэю
\begin{equation}\label{eq22}
\lim\limits_{k\rightarrow\infty}(n-1)\frac{\log\log
M_f(R_{m_k})}{\log R_{m_k}}\ge \frac{\log 2}{\log (K\theta)}>0\,.
\end{equation}
═рь юёЄрыюё№ яюърчрЄ№, ўЄю эш фы  ъръющ фЁєующ яюёыхфютрЄхы№эюёЄш
$x_k,$ фы  ъюЄюЁющ яЁхфхы
$\lim\limits_{k\rightarrow\infty}(n-1)\frac{\log\log M_f(x_k)}{\log
x_k}$ ёє∙хёЄтєхЄ, юэ эх ьюцхЄ с√Є№ ьхэ№°х тхышўшэ√, ёЄю ∙хщ т яЁртющ
ўрёЄш (\ref{eq22}). ╚Єръ, яєёЄ№ $x_k$ -- Єрър  яюёыхфютрЄхы№эюёЄ№,
Єюуфр яю шэфєъЎшш яюёЄЁюшь яюфяюёыхфютрЄхы№эюёЄш эюьхЁют $k_l$ ш
$m_l$ Єръшх, ўЄю $K^{m_l}\theta^{m_l}R\le x_{k_l}\le
K^{m_l+1}\theta^{m_l+1}R.$ ═х юуЁрэшўштр  юс∙эюёЄш, ьюцэю ёўшЄрЄ№,
ўЄю яЁхфхы $\lim\limits_{l\rightarrow\infty}(n-1)\frac{\log\log
M_f(K^{m_l}\theta^{m_l}R)}{\log K^{m_l}\theta^{m_l}R}$ Єръцх
ёє∙хёЄтєхЄ. ┬ ¤Єюь ёыєўрх, ттшфє (\ref{eq22}) ь√ яюыєўшь, ўЄю
$$\lim\limits_{l\rightarrow\infty}(n-1)\frac{\log\log M_f(x_{m_l})}{\log
x_{m_l}}\ge \lim\limits_{l\rightarrow\infty}(n-1)\frac{\log\log
M_f(K^{m_l}\theta^{m_l}R)}{\log K^{m_l}\theta^{m_l}R+\log
K\theta}=\frac{\log 2}{\log (K\theta)}\,.$$ ╥хюЁхьр \ref{th3}
яюыэюёЄ№■ фюърчрэр. $\Box$

\medskip
\section{═хёъюы№ъю ёыют ю ёЁртэхэшш Ёхчєы№ЄрЄют ЁрсюЄ√ ё сюыхх Ёрээшьш
Ёхчєы№ЄрЄрьш} ═х тфртр ё№ т яюфЁюсэ√щ ёЁртэшЄхы№э√щ рэрышч
яюыєўхээ√ї т эрёЄю ∙хщ ёЄрЄ№х єЄтхЁцфхэшщ яю юЄэю°хэш■ ъ Ёхчєы№ЄрЄрь
эхфртэю т√°хф°хщ ЁрсюЄ√ ╩.~╨рщрыр \cite{Ra$_1$}, єърцхь эр ёыхфє■∙шх
юсёЄю Єхы№ёЄтр, яюфў╕Ёъштр■∙шх шї эютшчэє. ┬ юЄышўшх юЄ єёыютш  Єшяр
(\ref{eq18}), т ЁрсюЄх \cite{Ra$_1$} юёэютэ√ь єёыютшхь рэрыюушўэюую
тшфр  ты хЄё  ЄЁхсютрэшх
$$\frac{1}{m(B(0, t))}\int\limits_{B(0, t)}
\exp(\Phi(K_O(x, f))\,dm(x)\le A\,,$$ уфх $\Phi:[0,
\infty)\rightarrow [0, \infty)$ -- ёЄЁюую тючЁрёЄр■∙р 
фшЇЇхЁхэЎшЁєхьр  ЇєэъЎш , Єрър  ўЄю
$\int\limits_{1}^{\infty}\frac{\Phi^{\,\prime}(t)}{t}\,dt=\infty,$
$t\cdot\Phi^{\,\prime}(t)\rightarrow \infty$ яЁш $t\rightarrow
\infty$ ш $\exp(\Phi(K_O(x, f)))\in L_{loc}^1({\Bbb R}^n).$ ╚ч
єърчрээ√ї єёыютшщ, т ўрёЄэюёЄш, т√ЄхърхЄ, ўЄю тэх°э   фшырЄрЎш 
юЄюсЁрцхэш  $K_O(x, f)$ фюыцэр с√Є№ ыюъры№эю ёєььшЁєхьр т ы■сющ
ёЄхяхэш $\alpha>0$ ш х╕ ёЁхфэхх чэрўхэшх, тч Єюх т ¤Єющ ёЄхяхэш, яю
ёъюы№ єуюфэю сюы№°шь °рЁрь фюыцэю с√Є№ юуЁрэшўхэю.

╬фэръю, эхЄЁєфэю яюёЄЁюшЄ№ яЁшьхЁ тэх°эхщ фшырЄрЎшш $Q:=K_O(x, f),$
фы  ъюЄюЁющ т√яюыэхэ√ єёыютш  (\ref{eq18}) ш (\ref{eq16}) (сюыхх
Єюую, $\widetilde{q}_{\alpha, 0}(r)\le C=const,$ уфх
$\widetilde{q}_{\alpha, 0}(r)$ -- ёЁхфэхх чэрўхэшх $Q^{\alpha}(x)$
яю ёЇхЁх $S(0, r),$ $\alpha>n-1$), юфэръю, т Єю цх тЁхь , ёЁхфэшх
чэрўхэш  юЄ ЇєэъЎшш $Q_{\varepsilon}(x):=Q^{\alpha\varepsilon}(x),$
$\varepsilon>1,$ яю °рЁрь ёъюы№ єуюфэю сюы№°юую Ёрфшєёр эх
юуЁрэшўхэ√. ─ы  ¤Єюую т ${\Bbb R}^n$ ЁрёёьюЄЁшь яЁш яЁюшчтюы№эюь
$r>0$ Єхыхёэ√щ єуюы $\gamma(r)$ ёшььхЄЁшўэ√щ юЄэюёшЄхы№эю ъръющ-ышсю
ЇшъёшЁютрээюую ыєўр $l,$ шёїюф ∙хую шч эєы . ╬сючэрўшь ўрёЄ№ ёЇхЁ√
$S(0, r),$ ъюЄюЁр  юуЁрэшўхэр фрээ√ь Єхыхёэ√ь єуыюь, ёшьтюыюь
$M(r).$ (╧юЄЁхсєхь фюяюыэшЄхы№эю, ўЄюс√ ЇєэъЎш  $\gamma(r)$ с√ыр
шчьхЁшьр яю $r$). ╧ю юяЁхфхыхэш■, $\gamma(r)=\frac{{\cal
H}^{n-1}(M(r))}{r^{n-1}}$ (т ўрёЄэюёЄш, $0\le {\cal
H}^{n-1}(M(r))\le \omega_{n-1}r^{n-1}$ ш
$0\le\gamma(r)\le\omega_{n-1}$). ╧юырурхь ЄхяхЁ№ фы  ЇшъёшЁютрээющ
шчьхЁшьющ яю ╦хсхує ЇєэъЎшш $\psi:(0, \infty)\rightarrow (0,
\infty)$ (Єръющ ўЄю $\psi(r)\equiv 1\quad \forall\,\, r\in (0, 1)$)
$Q^{\alpha}(x):=\psi(|x|)$ яЁш $x\in M(r),$ $|x|=r,$ ш
$Q^{\alpha}(x)=0$ т яЁюЄштэюь ёыєўрх. ╥юуфр $\widetilde{q}_{\alpha,
0}(r)=\frac{1}{\omega_{n-1}}\gamma(r)\psi(r),$ $$\frac{1}{m(B(0,
R))}\int\limits_{B(0,
R)}Q^{\alpha}(x)dm(x)=\frac{1}{\Omega_nR^n}\int\limits_{0}^R
\psi(r)\gamma(r)r^{n-1}dr$$ ш $$\frac{1}{m(B(0,
R))}\int\limits_{B(0,
R)}Q^{\alpha\varepsilon}(x)dm(x)=\frac{1}{\Omega_nR^n}\int\limits_{0}^R
\psi^{\varepsilon}(r)\gamma(r)r^{n-1}dr\,.$$ ╧юырурхь ЄхяхЁ№ яЁш
$r\ge 1:$ $\psi(r)=r,$ $\gamma(r)=\omega_{n-1}/r.$ ╠√ тшфшь т ¤Єюь
ёыєўрх, ўЄю $\widetilde{q}_{\alpha, 0}(r)\le C_1=const,$
$\frac{1}{m(B(0, R))}\int\limits_{B(0, R)}Q^{\alpha}(x)dm(x)\le
C_2=const,$ юфэръю, $\frac{1}{m(B(0, R))}\int\limits_{B(0,
R)}Q^{\alpha\varepsilon}(x)dm(x)\rightarrow\infty$ яЁш $R\rightarrow
\infty.$ ╥ръшь юсЁрчюь, яюёЄЁюхэ яЁшьхЁ тэх°эхщ фшырЄрЎшш $K_O(x,
f),$ єфютыхЄтюЁ ■∙шщ яЁш $Q:=K_O(x, f)$ єёыютш ь (\ref{eq18}) ш
(\ref{eq16}), эю эх єфютыхЄтюЁ ■∙шщ ЄЁхсютрэш■ $\frac{1}{m(B(0,
t))}\int\limits_{B(0, t)} \exp(\Phi(K_O(x, f))\,dm(x)\le A=const$
($\Phi:[0, \infty)\rightarrow [0, \infty)$ -- яЁюшчтюы№эр  ёЄЁюую
тючЁрёЄр■∙р  фшЇЇхЁхэЎшЁєхьр  ЇєэъЎш , Єрър  ўЄю
$t\cdot\Phi^{\,\prime}(t)\rightarrow \infty$ яЁш $t\rightarrow
\infty$).

\medskip
╨хчєы№ЄрЄ√ ЁрсюЄ√ ьюуєЄ яЁшьхэхэ√ ъ Ёрчышўэ√ь ъырёёрь яыюёъшї ш
яЁюёЄЁрэёЄтхээ√ї юЄюсЁрцхэшщ (ёь., эряЁ., \cite{IM},
\cite{MRSY},\cite{MRSY$_1$}, \cite{GRSY}, \cite{BGMV} ш \cite{GS}).


\medskip
╩╬═╥└╩╥═└▀ ╚═╘╬╨╠└╓╚▀

\noindent{{\bf ┼тухэшщ └ыхъёрэфЁютшў ╤хтюёЄ№ эют} \\╚эёЄшЄєЄ
яЁшъырфэющ ьрЄхьрЄшъш ш ьхїрэшъш ═└═ ╙ъЁршэ√ \\
83 114 ╙ъЁршэр, у. ─юэхЎъ, єы. ╨юч√ ╦■ъёхьсєЁу, ф. 74, \\
Ёрс. Єхы. 8-380-62-311 01 45, \\
e-mail: brusin2006@rambler.ru, esevostyanov2009@mial.ru}

\end{document}